\documentclass{IEEEtran}
\usepackage[latin9]{inputenc}
\usepackage{array}
\usepackage{amsmath}
\usepackage{amsthm}
\usepackage{amssymb}
\usepackage{graphicx}
\usepackage[unicode=true,
 bookmarks=false,
 breaklinks=false,pdfborder={0 0 1},backref=section,colorlinks=false]
 {hyperref}

\makeatletter

\providecommand{\tabularnewline}{\\}

\theoremstyle{plain}
\newtheorem{thm}{\protect\theoremname}
\theoremstyle{plain}
\newtheorem{prop}[thm]{\protect\propositionname}
\theoremstyle{remark}
\newtheorem{rem}[thm]{\protect\remarkname}
\theoremstyle{plain}
\newtheorem{lem}[thm]{\protect\lemmaname}

\usepackage{graphics}
\usepackage{epsfig}
\usepackage{times}
\usepackage{color}
\usepackage{array}
\usepackage{accents}

\usepackage{stackengine}
\usepackage{url}
\usepackage{enumerate}

\usepackage{cite}

\makeatother

\providecommand{\lemmaname}{Lemma}
\providecommand{\propositionname}{Proposition}
\providecommand{\remarkname}{Remark}
\providecommand{\theoremname}{Theorem}

\begin{document}
\title{Non-overshooting continuous in convergence sliding mode control of second-order systems}
\author{Michael Ruderman and Denis Efimov\thanks{M. Ruderman is with University of Agder. Grimstad, Norway. Email:
\texttt{\small{}michael.ruderman@uia.no}} \thanks{D. Efimov is with Inria, Univ. Lille, CNRS, CRIStAL. Lille, France.}
\thanks{This work was supported by RCN grant number 340782.}}

\maketitle
\thispagestyle{empty}
\begin{abstract}
This paper proposes a novel nonlinear sliding
mode state feedback controller for perturbed second-order systems.
In analogy to a linear proportional-derivative (PD) feedback
control, the proposed nonlinear scheme uses the output of interest
and its time derivative. The control has only one free design
parameter, and the closed-loop system is shown to possess uniform
boundedness and finite-time convergence of trajectories in the
presence of matched disturbances. We derive a strict Lyapunov
function for the closed-loop control system with a bounded
exogenous perturbation, and use it for both, the control parameter
tuning and analysis of the finite-time convergence. The essential
features of the proposed new control law is non-overshooting
despite the unknown dynamic disturbances and the continuous
control action during the convergence to zero equilibrium. Apart
from the numerical results, a revealing experimental example is
also shown in favor of the proposed control and in comparison with
PD and sub-optimal nonlinear damping regulators.
\end{abstract}


\section{Introduction}

\label{sec:1}

Convergence properties of a transient response and the ability to
compensate for unknown (external) disturbances are often the
trade-offs when designing a feedback control at large. Well known,
a PD regulator, which is equivalent to the state-feedback control
for second-order systems, is often sufficient for the unperturbed
case provided that an exponential convergence is acceptable. When
the system behavior is subject also to static errors, an integral
control action can be added, thus resulting in a standard PID (or
integral-plus-state feedback) scheme in case of the linear
feedback regulation \cite{aastrom2006}. From another side, an
integral feedback action increases inherently the system order and
can also lead to the wind-up effects and even destabilization
\cite{hippe2006}. If disturbances are not static and possess
certain regularity only (e.g., Lipschitz continuity or/and
boundedness), then the robust feedback regulators like, for
instance, high-gain (see, e.g., \cite{ilchmann2008} and references
therein) or finite-time (see, e.g., \cite{bhat2000} and a recent
text \cite{efimov2021}) controllers can be required. Beyond that,
the input-to-state stability (ISS) theory \cite{sontag2008} is
usually applied for studying the stability and robustness in
presence of external disturbances, once the overall feedback
system is nonlinear \cite{khalil2002}.

Second-order systems capture a relatively large class of the
(perturbed) dynamic processes, especially in mechanical,
electromagnetic, and fluid-flow domains. In that way, a perturbed
double-integrator represents an often encountered system process.
A comparative overview of different feedback control principles
applied to a classical double-integrator can be found in
\cite{rao2001}. Also the sliding-mode control methods, see
\cite{shtessel2014,utkin2020}, become relatively popular for
compensating the unknown matched perturbations. Indeed, a feedback
control action proportional to the sign of the control error has
already been recognized in earlier works as particularly efficient
for a time-optimal stabilization \cite{fuller1960,haimo1986}, and
disturbance compensation \cite{Tsypkin1984relay}. However, a
discontinuous control signal, especially a continuously switching
one, can be undesirable in some scenarios, especially with regard
to the wear effects, overloading of the actuators, noise, and
energy consumption. To this end, in \cite{levant2005} the class of
quasi-continuous high order sliding mode control algorithms has
been introduced, which have the discontinuity at the origin only.
Therefore, a feedback control law that would allow both, a sign
driven robustness to the unknown but bounded perturbations and (in
addition) a continuous damping action which would minimize the
overshooting and steady-state oscillations, can be desirable. The
idea of an inverse to the output's norm damping was originally
proposed for second-order systems in the so-called sub-optimal
nonlinear damping control \cite{ruderman2021a}. Recently, the
sub-optimal nonlinear damping control was augmented by a linear
integral term in \cite{efimovruderman2024}, also with additional
analysis so as to cope with the matched constant disturbances.
Moreover, the ISS property was established for time-varying
bounded perturbations in that augmented control system
\cite{efimovruderman2024}.

The present work builds up on the idea proposed in
\cite{ruderman2021a} and uses the sign-feedback principle of
sliding mode control theory following \cite{levant2005}, while
introducing a novel nonlinear state feedback control for the
perturbed second-order systems. The unknown time-varying
perturbations are assumed to be upper bounded, and the controller
has only one free design parameter, similarly to
\cite{ruderman2021a}. The introduced nonlinear feedback controller
has the same structure as its linear counterpart (i.e., the PD
one) and feeds back the output of interest and its time
derivative. Moreover, it offers finite-time convergence despite
unknown bounded perturbations and is continuous during convergence
to the globally stable equilibrium at the origin, where the
sliding mode is originated (a signature behavior for quasi-
continuous sliding mode control laws). Particularly remarkable in
comparison to several robust second-order sliding mode controllers
is that the proposed control offers convergence without
overshooting the origin (or reference value) that comes in favor
of numerous applications. In addition, even if the disturbance
exceeds the admissible bounds in the steady-state equilibrium at
the origin, the designed controller preserves its invariance.

The paper is organized as follows. The proposed nonlinear state
feedback control is introduced in Section \ref{sec:2}, also
demonstrating qualitatively the phase portrait and shape of the
converging state trajectories and control value, and that in
comparison with a linear state feedback control (equivalent to PD)
and sub-optimal nonlinear damping control \cite{ruderman2021a}.
The main part of the closed-loop system analysis is delivered in
Section \ref{sec:3}. First, the solutions of the system are shown
to be uniformly (in disturbances) bounded and asymptotically
converging, then the finite-time convergence is formally proved,
and boundedness of the control is established. Section \ref{sec:4}
reports the experimental evaluation of the proposed control and
compares it with an equivalent PD and sub-optimal nonlinear
damping controllers. The results reported in this communication
are briefly summarized in Section \ref{sec:5}.

\section{Novel nonlinear feedback control}
\label{sec:2}

In this section, first, we describe the studied control problem;
second, the proposed control algorithm is presented; third, we
finish with an exemplary illustration of the dynamic behavior of
the closed-loop system.

\subsection{Problem statement}

We consider a class of perturbed second-order systems:
\begin{equation}
\frac{\mathrm{d}^{2}y(t)}{\mathrm{d}t^{2}}=u(t)+d(t),\quad
t\geq0,\label{eq:system}
\end{equation}
with the control signal $u(t)\in\mathbb{R}$. The matched unknown
disturbance quantity $d(t)\in\mathbb{R}$ is assumed to be Lebesgue
measurable and essentially bounded, where
$\|d\|_{\infty}=\text{ess}\sup_{t\geq0}|d(t)|\leq D$ with a known
upper bound $D>0$. Both dynamic system states,
$x(t)=\left(x_{1}(t),x_{2}(t)\right)^{\top}\equiv\bigl(y(t),\dot{y}(t)\bigr)^{\top}\in\mathbb{R}^{2}$
(including output derivative\footnote{We will use the doted
symbol, i.e., $\dot{a}\equiv\mathrm{d}a/\mathrm{d}t$, to denote
the time derivative of a variable $a$, equally as for the
first-order differential equations.}), are available for a
feedback design.

Our primary goal is a variant of robust finite-time stabilization
of \eqref{eq:system} at $x=0$: it is required to design a state
feedback $u(t)=u\bigl(x_{1}(t),x_{2}(t)\bigr)$ guaranteeing that
there exist $\sigma,T\in\mathcal{K}$ such that
$|x(t)|\leq\sigma(|x(0)|)$ for all $t\geq0$ and $|x(t)|=0$ for all
$t\geq T(|x(0)|)$ for any
$x(0)\in\mathbb{X}=\{x\in\mathbb{R}^{2}\,:\,x_{1}\neq0\}\cup\{0\}$
and all $\|d\|_{\infty}\leq D$. Thus, we assume that at the
initial time instant, the $x_{1}(0)$ output state is not zero, and
it is required to bring $x(t)$ to zero in a finite time
independently of the presence of a properly bounded disturbance
$d$.

Note that in case of an additional (inertial) scaling constant in
the left hand-side of \eqref{eq:system}, like in the experimental
study of a moving lumped mass shown in Section \ref{sec:4}, this
should be additionally incorporated as a normalization factor of
the control terms, cf. \cite{ruderman2022}. A more general setting
with an uncertain control coefficient, which is separated from
zero with known bounds, can be treated by properly augmenting the
control gain as it is usually done in the sliding mode theory.
This case is skipped for brevity.

\subsection{Control law}

The proposed control for $x(t)\in\mathbb{X}$ has the form:
\begin{equation}
u(t)=\begin{cases}
-|x_{1}(t)|^{-1}\bigl(\gamma x_{1}(t)+|x_{2}(t)|x_{2}(t)\bigr) & x_{1}(t)\ne0\\
-\gamma\,\text{sign}\bigl(x_{1}(t)\bigr) & x(t)=0
\end{cases},\label{eq:ctrl}
\end{equation}
where $\gamma>D$ is the only tuning parameter, and it yields for
$x(t)\in\mathbb{X}$ the closed-loop system dynamics to
\begin{gather}
\frac{\mathrm{d}x_{1}(t)}{\mathrm{d}t}=x_{2}(t),\label{eq:CL_syst}\\
\frac{\mathrm{d}x_{2}(t)}{\mathrm{d}t}=\qquad\;\;\nonumber \\
\begin{cases}
-|x_{1}(t)|^{-1}\bigl(\gamma x_{1}(t)+|x_{2}(t)|x_{2}(t)\bigr)+d(t) & x_{1}(t)\ne0\\
-\gamma|x_{1}(t)|^{-1}x_{1}(t)+d(t) & x_{1}(t)=0
\end{cases}.\nonumber
\end{gather}
As we can conclude, the proposed control is discontinuous at
$x_{1}=0$, and in the set $\mathbb{X}$ the irregularity happens at
the origin only, which is the equilibrium of the system. Indeed,
there are two discontinuous terms. The first one is the
conventional sign function
$x_{1}(t)|x_{1}(t)|^{-1}=\text{sign}(x_{1}(t))$ widely used in the
sliding mode control theory, where discontinuity appears for
$x_{1}(t)=0$ only. However, in the set $\mathbb{X}$, having
$x_{1}(t)=0$ is equivalent to $x(t)=0$. Consequently, this term
introduces discontinuity just at the equilibrium, and for
definition of solutions in the face of this discontinuity, the
Filippov's theory can be used with convex embedding
\cite{shtessel2014,utkin2020}. The second discontinuity may come
from the fraction $|x_{2}(t)|x_{2}(t)|x_{1}(t)|^{-1}$, but this
term is active for $x_{1}(t)\ne0$ only, and it stays always
bounded on the trajectories of \eqref{eq:CL_syst}. The
investigation of the influence of this term requires a special
attention at $x_{1}=0$, so that a proper analysis approach will be
utilized in the following. Below, in order to substantiate
existence and uniqueness of solutions for \eqref{eq:CL_syst}, we
will demonstrate similarity of the system \eqref{eq:CL_syst} via a
time change to another system with well defined solutions, as we
have already done for the sub-optimal nonlinear damping control in
\cite{efimovruderman2024}. Next, the stability and performance of
the closed-loop system can be analyzed. Before doing it, we
proceed with some illustrative simulation results showing the
intuition behind the proposed design and qualitative properties of
the state trajectories of the controlled system.

\subsection{Exemplary}

An illustration of phase-portrait of the closed-loop control
system \eqref{eq:CL_syst} is given in Figure \ref{fig:2:0}. Here
the control gain factor $\gamma=100$ is arbitrary assigned. Due to
a symmetric behavior with respect to the origin, the trajectories
for different initial conditions are shown in the 1st and 4th
quadrants only.
\begin{figure}[!h]
\centering
\includegraphics[width=0.98\columnwidth]{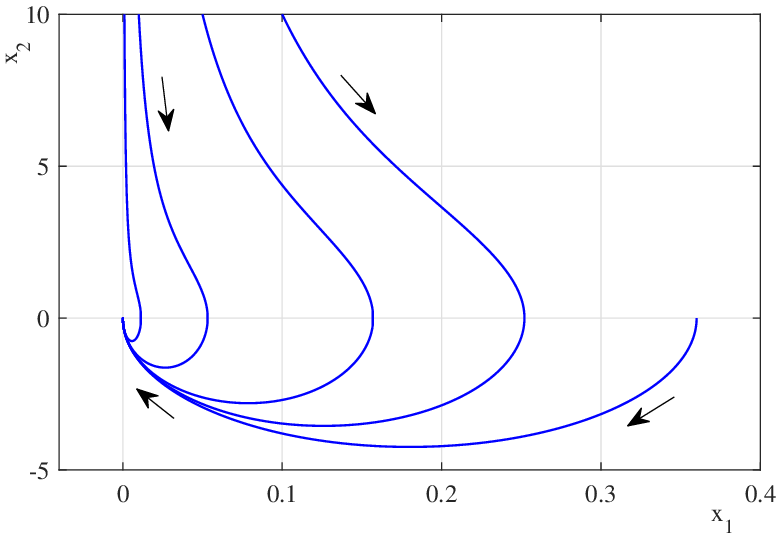}
\caption{\label{fig:2:0} Phase portrait exemplary of the
closed-loop system \eqref{eq:CL_syst}.}
\end{figure}

The trajectories of $x_{2}(t)$ and the control values $u(t)$ are
also qualitatively compared in Figure \ref{fig:2:1} between the
proposed control \eqref{eq:CL_syst}, sub-optimal nonlinear damping
control \cite{ruderman2021a}, and critically damped linear state
feedback control (equivalent to PD one).
\begin{figure}[!h]
\centering
\includegraphics[width=0.98\columnwidth]{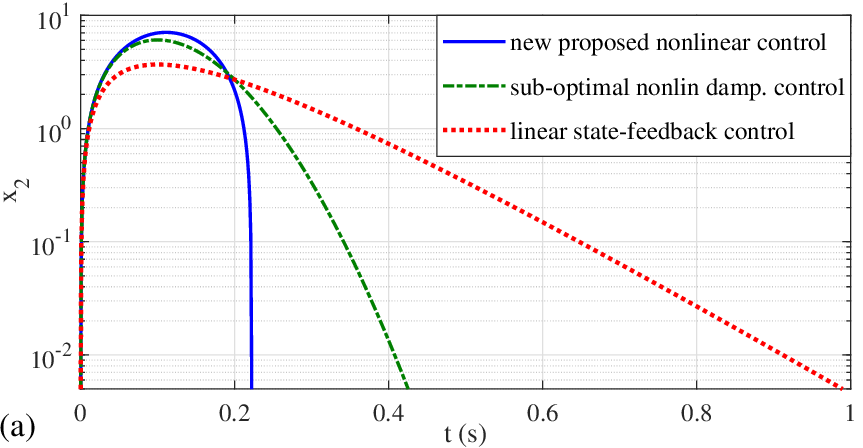}
\includegraphics[width=0.98\columnwidth]{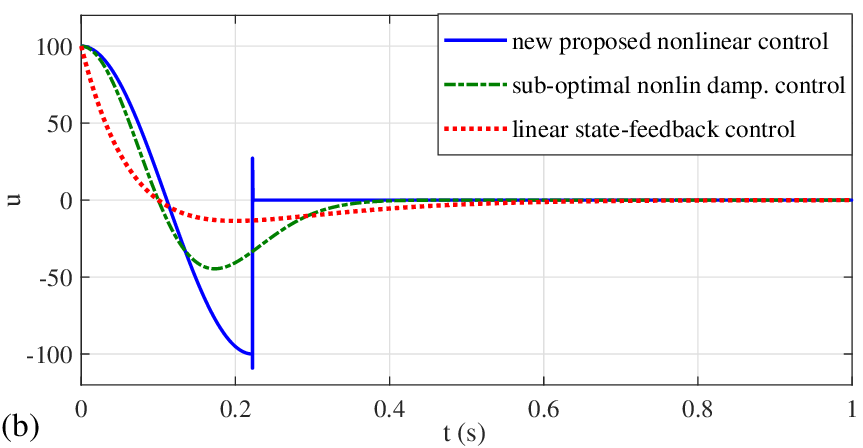}
\caption{\label{fig:2:1} Response of the dynamic state $x_{2}(t)$
in (a) and control signal $u$ in (b) of the proposed new nonlinear
control \eqref{eq:CL_syst}, sub-optimal nonlinear damping control
\cite{ruderman2021a}, and critically damped linear state feedback
control.}
\end{figure}
All three controllers are assigned with the same output feedback
gain factor $\gamma=100$, while the response of $x_{2}(t)$ is
depicted on the logarithmic scale for the sake of a better
comparison, see Figure \ref{fig:2:1} (a). A finite time
convergence of the control \eqref{eq:CL_syst} can be qualitatively
recognized. The corresponding control values $u(t)$ are shown in
Figure \ref{fig:2:1} (b). All three control laws disclose the same
value at the beginning, this is due to the equal gaining factor
$\gamma$ and the same initial conditions $x_{1}(0)=-1$,
$x_{2}(0)=0$. Then, one can recognize that the control
\eqref{eq:CL_syst} uses the range of admissible values more
efficiently during the convergence phase.

\section{Analysis}

\label{sec:3}

We start this section with the proof of existence and uniqueness
of solutions in \eqref{eq:CL_syst}. Next, the uniform boundedness
of solutions and asymptotic convergence to the origin will be
substantiated (uniformity is understood as independence of these
properties on the properly bounded inputs). The rate of
convergence is evaluated at the end of the section, together with
the upper bound on the control magnitude.

\subsection{Existence of solutions}

\label{sec:3:sub:1}

In the Appendix, for the system \eqref{eq:aux} operating in time
$\tau\geq0$ the following properties are demonstrated for any
$x(0)\in\mathbb{R}^{2}$ and the inputs with $\|d\|_{\infty}\leq
D$:
\begin{enumerate}
\item global existence and uniqueness of solutions $x(\tau)=x(\tau,x(0),d)$
for all $\tau\geq0$,
\item invariance of the sets $\mathcal{X}=\{x\in\mathbb{R}^{2}:x_{1}=0\}$
and
$\mathbb{X}=\left(\mathbb{R}^{2}\setminus\mathcal{X}\right)\cup\{0\}$,
\item existence of the equilibrium at the origin, which is uniformly globally
asymptotically stable.
\end{enumerate}
To use these facts in the analysis of \eqref{eq:CL_syst},
introduce a new (scaled) time argument for \eqref{eq:aux}:
\begin{equation}
t:=\phi(\tau)=\int\limits
_{0}^{\tau}\bigl|x_{1}(s)\bigr|ds,\label{eq:newtime}
\end{equation}
which is well-defined, and $t\in\mathbb{R}_{+}$ for
$\tau\in\mathbb{R}_{+}$ (indeed, $t=0$ for $\tau=0$, and
$\lim_{\tau\to+\infty}\int\limits
_{0}^{\tau}\bigl|x_{1}(s)\bigr|ds\leq+\infty$). Moreover,
$\phi:\mathbb{R}_{+}\to\mathbb{R}_{+}$ is a monotonously growing
function for $x(0)\in\mathbb{X}\setminus\{0\}$ since the origin in
such a case is reached for $\tau\to+\infty$ only, and it has an
inverse $\tau=\phi^{-1}(t)$. However, the domain of values of $t$
may be bounded since $\bigl|x_{1}(\tau)\bigr|\to0$ as
$\tau\to+\infty$. That implies
$\mathrm{d}t=\bigl|x_{1}(\tau)\bigr|\mathrm{d}\tau$ while
transforming \eqref{eq:aux} to the closed-loop system
\eqref{eq:CL_syst} for $x(0)\in\mathbb{X}\setminus\{0\}$. Indeed,
with a slight ambiguity define $x(t)=x(\phi^{-1}(t))=x(\tau)$ and
calculate its derivative in the new time $t$ keeping in mind that
$\frac{\mathrm{d}}{\mathrm{d}t}\phi^{-1}(t)=\Bigl|x_{1}\bigl(\phi^{-1}(t)\bigr)\Bigr|^{-1}$,
which gives that $x(t)$ is a solution of \eqref{eq:CL_syst} since
\[
\frac{\mathrm{d}x(t)}{\mathrm{d}t}=\Bigl|\frac{\mathrm{d}x(\tau)}{\mathrm{d}\tau}\Bigr|_{\tau=\phi^{-1}(t)}\frac{\mathrm{d}\phi^{-1}(t)}{\mathrm{d}t}.
\]
Therefore, the systems \eqref{eq:aux} and \eqref{eq:CL_syst} have
the same solutions in $\mathbb{X}\setminus\{0\}$ (and the origin
is the equilibrium for both systems), where the only difference is
that each solution becomes scaled along the independent time
variable based on the norm of $x_{1}$.

Since \eqref{eq:aux} has a unique solution asymptotically
approaching the origin for any initial condition
$x(0)\in\mathbb{X}$ and input with $\|d\|_{\infty}\leq D$, the
system \eqref{eq:CL_syst} also admits this solution while
demonstrating the same behavior in finite or infinite time. In the
former case, since the origin is the equilibrium of
\eqref{eq:CL_syst}, the zero solution can be extended for all
$t\geq0$. Therefore, the following result has been proven:
\begin{prop}
The system \eqref{eq:CL_syst} admits an unique solution for any
$x(0)\in\mathbb{X}$ and $\|d\|_{\infty}\leq D$ well defined with
$x(t)\in\mathbb{X}$ for all $t\geq0$.
\end{prop}
In fact, via such a similarity of \eqref{eq:aux} and
\eqref{eq:CL_syst} through the time scaling \eqref{eq:newtime}, a
stronger property is substantiated: boundedness of solutions and
asymptotic convergence to the origin for \eqref{eq:CL_syst} in
$\mathbb{X}$ uniformly in properly bounded $d$. However, below we
would like to apply the Lyapunov function method directly to
\eqref{eq:CL_syst}, so as to establish this property with
demonstration of the finite-time convergence rate.

\subsection{Analysis of stability}

\label{sec:3:sub:2}

We are in the position to prove the uniform boundedness of
solutions and asymptotic convergence to the origin for the
perturbed system \eqref{eq:CL_syst} in $\mathbb{X}$.

\vspace{1mm}

\begin{thm}
\label{thm:GAS} The origin for the system \eqref{eq:CL_syst} is
uniformly attractive for any $x(0)\in\mathbb{X}$ and
$\|d\|_{\infty}\leq D$ if
\begin{equation}
\gamma>D^{1.5}+D+\frac{1}{2},\label{eq:cond}
\end{equation}
and the solutions are bounded for all $t\geq0$.
\end{thm}
\vspace{2mm}

\begin{proof}
As in the Appendix for the system \eqref{eq:aux}, consider for
\eqref{eq:CL_syst} a strict Lyapunov function candidate
\begin{equation}
V(x)=\frac{1}{2}\,\left(\begin{array}{c}
z_{1}\\
x_{2}
\end{array}\right)^{\top}\,\left(\begin{array}{cc}
2\,\gamma & \varepsilon\\
\varepsilon & 1
\end{array}\right)\left(\begin{array}{c}
z_{1}\\
x_{2}
\end{array}\right)\label{eq:lyapunov}
\end{equation}
with $z_{1}=\sqrt{|x_{1}|}\,\mathrm{sign}(x_{1})$, which is
positive definite for $\sqrt{2\gamma}>\varepsilon>0$. This
Lyapunov function is continuously differentiable for
$x\in\mathbb{X}\setminus\{0\}$. Taking the time derivative of $V$
on this set one obtains
\begin{eqnarray*}
\dot{V} & = & x_{2}\,d-x_{2}^{2}\frac{|x_{2}|}{|x_{1}|}-\varepsilon\bigl(\gamma-\mathrm{sign}(x_{1})\,d\bigr)\,\sqrt{|x_{1}|}\\
 &  & +\,\varepsilon\frac{x_{2}^{2}}{\sqrt{|x_{1}|}}\,\Bigl(0.5-\mathrm{sign}(x_{1})\mathrm{sign}(x_{2})\Bigr).
\end{eqnarray*}
Using the \emph{Young's inequality} \cite{young1912} which
postulates
\[
ab\leq\frac{1}{p}\,a^{p}+\frac{p-1}{p}\,b^{p/(p-1)}
\]
for any $a,b\in\mathbb{R}_{+}$ and $p>1$, one can show that
\begin{equation}
\dot{V}\leq-\biggl(\frac{2}{3}-\varepsilon\biggr)\,x_{2}^{2}\frac{|x_{2}|}{|x_{1}|}-\biggl[\varepsilon\Bigl(\gamma-\frac{1}{2}-D\Bigr)-\frac{2}{3}D^{1.5}\biggr]\,\sqrt{|x_{1}|}\label{eq:dV}
\end{equation}
for any $x\in\mathbb{X}$ and $|d|\leq D$. It can be seen, that
taking
\[
\frac{2}{3}\frac{D^{1.5}}{\gamma-\frac{1}{2}-D}<\varepsilon<\frac{2}{3}
\]
both expressions in the brackets on the right hand-side of
$\dot{V}$ are positive under \eqref{eq:cond}. Hence,
$\dot{V}(x,d)<0$ for $x\in\mathbb{X}\setminus\{0\}$ and $|d|\leq
D$ that completes the proof.
\end{proof}
\vspace{2mm}

\begin{rem}
A consequence of the result given in Theorem \ref{thm:GAS}, which
establishes forward invariance of the set $\mathbb{X}$ and global
convergence of all solutions in \eqref{eq:cond} to the origin, is
that it implies the absence of overshooting: the trajectories
cannot cross the line $x_{1}=0$, and they reach it only at the
origin.
\end{rem}
The low bound of the control gain $\gamma$, which is satisfying
the parametric condition \eqref{eq:cond}, is shown as a function
of the maximal disturbance magnitude, i.e., $D$, in Figure
\ref{fig:3:1}.
\begin{figure}[!h]
\centering
\includegraphics[width=0.98\columnwidth]{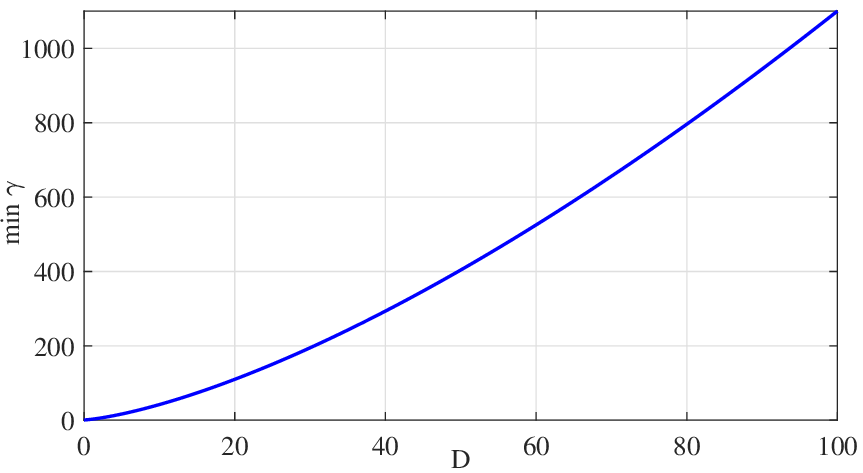}
\caption{\label{fig:3:1} Low bound of control gain depending on
perturbation bound.}
\end{figure}

\subsection{Finite-time convergence and control boundedness}

\label{sec:3:sub:3}

In order to prove the finite-time convergence (see, e.g.,
\cite{bhat2000} for details) of the system \eqref{eq:CL_syst},
which can be supposed from the numerical results given in Figure
\ref{fig:2:1} (a), we make further use of the strict Lyapunov
function $V$ given in \eqref{eq:lyapunov}.

\vspace{2mm}

\begin{thm}
\label{thm:FTC} The system \eqref{eq:CL_syst} is finite-time
convergent to the origin for any $x(0)\in\mathbb{X}$ and
$\|d\|_{\infty}\leq D$ provided that \eqref{eq:cond} is satisfied.
\end{thm}
\begin{proof}
The estimate \eqref{eq:dV} implies boundedness of solutions of
\eqref{eq:CL_syst} in $\mathbb{X}$ for the inputs with magnitude
smaller than $D$, with their asymptotic convergence to the origin.
Let us consider an auxiliary variable
$\zeta(t)=x_{2}(t)x_{1}(t)^{-1}$, which is well defined on the
solutions of the system \eqref{eq:CL_syst} (in the set
$\mathbb{X}\setminus\{0\}$), and it has the dynamics:
\begin{gather*}
\dot{\zeta}=\frac{d\,x_{1}-\bigl(\gamma|x_{1}|+\text{sign}(x_{1})|x_{2}|x_{2}\bigr)-x_{2}^{2}}{x_{1}^{2}}\\
=\frac{d\,\text{sign}(x_{1})-\gamma}{|x_{1}|}-|\zeta|\zeta-\zeta^{2}\\
=\frac{d\,\text{sign}(x_{1})-\gamma}{|x_{1}|}-\begin{cases}
2\zeta^{2} & \zeta>0\\
0 & \zeta\leq0
\end{cases}.
\end{gather*}
Note that $(d\,\text{sign}(x_{1})-\gamma)|x_{1}|^{-1}<0$ due to
\eqref{eq:cond}, which implies that for any $x(0)\in\mathbb{X}$
and $\|d\|_{\infty}\leq D$ there is $t_{1}\geq0$ such that
$\zeta(t)<0$ for all $t>t_{1}$, i.e., the trajectories enter in
the second or the fourth quadrants after finite transients, where
the variable $\zeta$ stays separated with zero. Next, consider
another auxiliary variable $z(t)=x_{2}^{2}(t)|x_{1}(t)|^{-1}$,
which admits the dynamics:
\begin{gather*}
\dot{z}=\frac{|x_{2}|}{|x_{1}|}\,\Bigl(2\text{sign}(x_{2})\bigl(d-\gamma\,\text{sign}(x_{1})\bigr)\\
-\bigl(2+\text{sign}(x_{1})\text{sign}(x_{2})\bigr)z\Bigr).
\end{gather*}
For $t>t_{1}$ and taking into account \eqref{eq:cond},
$\text{sign}(x_{2})(d-\gamma\text{sign}(x_{1}))>0$ out the origin
(the lines $x_{1}=0$ and $x_{2}=0$ can be reached only there).
Hence, for any $x(0)\in\mathbb{X}$ and $\|d\|_{\infty}\leq D$,
there is $t_{2}\geq0$ such that $z(t)\geq\theta>0$ for all $t\geq
t_{2}$ (while the trajectories do not converge to zero), where
$\theta$ depends only on the choice of $\gamma$ and $D$. So, there
are two possibilities, either the trajectories approach the origin
in a finite time, or the variable $z(t)$ becomes strictly positive
and separated with zero, which returning to the Lyapunov function
$V(x)$ from \eqref{eq:lyapunov} and the estimate \eqref{eq:dV}
implies for $t\geq t_{2}$:
\begin{gather*}
\dot{V}\leq-\biggl(\frac{2}{3}-\varepsilon\biggr)\,\theta|x_{2}|-\biggl[\varepsilon\Bigl(\gamma-\frac{1}{2}-D\Bigr)-\frac{2}{3}D^{1.5}\biggr]\,\sqrt{|x_{1}|}\\
\leq-\kappa\sqrt{V}
\end{gather*}
for some $\kappa>0$. The latter differential inequality ensures a
finite-time convergence of $V$ to zero according to
\cite{bhat2000}.
\end{proof}
Slightly extending the arguments used in the last proof we can
substantiate the boundedness of the control \eqref{eq:ctrl}:
\begin{prop}
In the conditions of Theorem \ref{thm:FTC}, the amplitude of
control \eqref{eq:ctrl} admits an upper estimate
\[
|u(t)|\leq\gamma+\max\{x_{2}^{2}(0)|x_{1}(0)|^{-1},2(D+\gamma)\},\;\forall
t\geq0.
\]
\end{prop}
\begin{proof}
Following the proof of Theorem \ref{thm:FTC} and considering an
auxiliary Lyapunov function $\mathcal{V}(z)=0.5z^{2}$ we get the
estimate on its time derivative for $t\geq0$ while $z(t)\ne0$
(until the origin is not reached, since it is the only invariant
solution on the set $z=0$):
\begin{gather*}
\dot{\mathcal{V}}=\frac{|x_{2}|}{|x_{1}|}\,\Bigl(2\text{sign}(x_{2})\bigl(d-\gamma\,\text{sign}(x_{1})\bigr)z\\
-\bigl(2+\text{sign}(x_{1})\text{sign}(x_{2})\bigr)z^{2}\Bigr)\\
\leq\frac{|x_{2}|}{|x_{1}|}\,\Bigl(2(D+\gamma)|z|-z^{2}\Bigr)\leq\frac{|x_{2}|}{|x_{1}|}\,\Bigl(2(D+\gamma)^{2}-\mathcal{V}\Bigr),
\end{gather*}
which implies the following boundedness estimate in the time
domain:
$\mathcal{V}(z(t))\leq\max\{\mathcal{V}(z(0)),2(D+\gamma)^{2}\}$
(the multiplier $\frac{|x_{2}|}{|x_{1}|}$ may be zero at the
isolated time instant $t_{1}$ when $x_{2}(t_{1})=0$ only), or
\begin{align*}
z(t) & \leq\max\{z(0),2(D+\gamma)\}
\end{align*}
for all $t\geq0$. Note that $|u(t)|\leq\gamma+z(t)$ for all
$t\geq0$ in \eqref{eq:ctrl} that gives the required upper bound
for the control.
\end{proof}
Moreover, since the set $z\leq2(D+\gamma)$ is attractive and
forward invariant (it contains the origin), it implies that for
any trajectory approaching the origin (even initially $x_{1}(0)$
was arbitrary small) the control has an upper bound
\[
|u|\leq2D+3\gamma,
\]
and the term $\frac{|x_{2}|x_{2}}{|x_{1}|}$ does not impose a
singularity in the system \eqref{eq:CL_syst}. In addition, for the
set of initial conditions with $x_{1}(0)\ne0$ and $x_{2}(0)=0$
(i.e., when the control is applied to the system \eqref{eq:system}
in a steady state for $u(t)=d(t)=0$, $t\leq0$), the amplitude of
\eqref{eq:ctrl} is bounded by $2D+3\gamma$ for all $t\geq0$.
\begin{rem}
The proven asymptotic boundedness of $z$ and also its separation
with zero, i.e., the forward attractiveness of the set $\theta\leq
z\leq2(D+\gamma)$ while the trajectory stays in
$\mathbb{X}\setminus\{0\}$, implies that \eqref{eq:CL_syst} admits
an invariant cone in the second and third quadrants.
\end{rem}

\subsection{Numerical simulations}

\begin{figure}[!h]
\centering
\includegraphics[width=0.98\columnwidth]{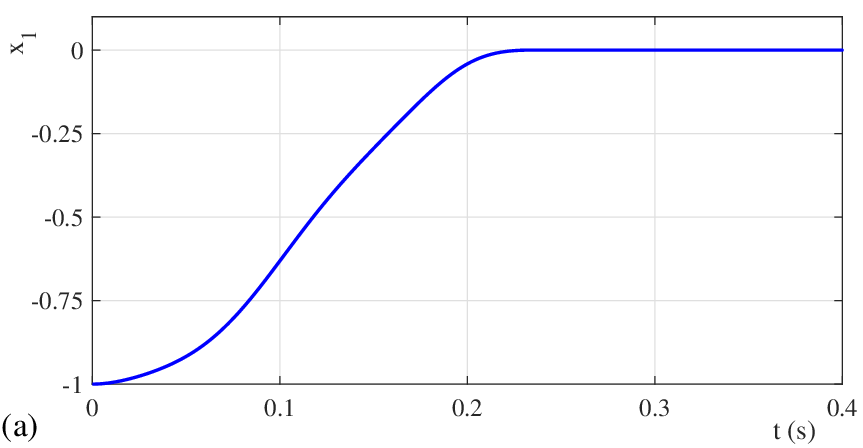}
\includegraphics[width=0.98\columnwidth]{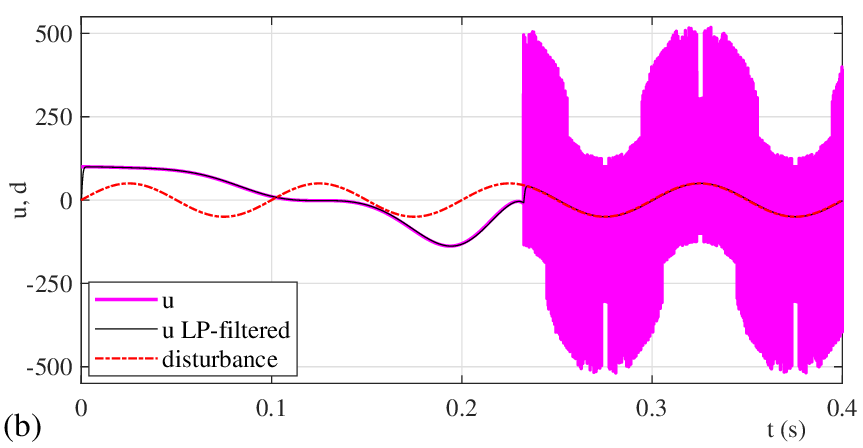}
\includegraphics[width=0.98\columnwidth]{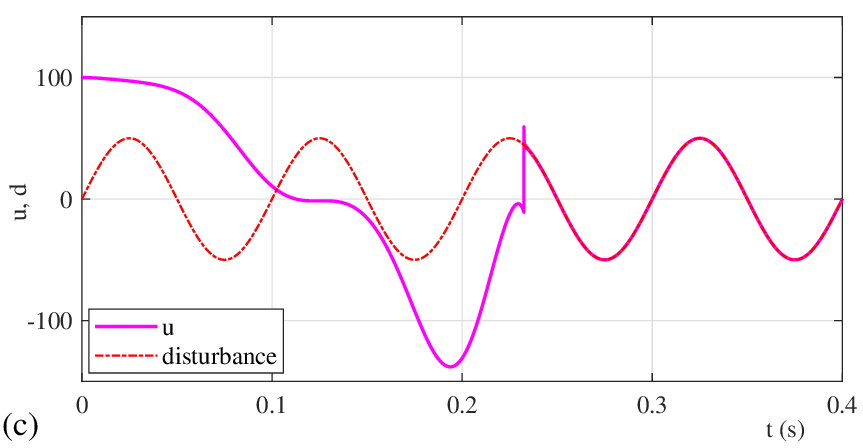}
\caption{Output response of the perturbed control system
\eqref{eq:CL_syst} in (a), and the control signal versus
disturbance value by using: thresholding scheme
\eqref{eq:contFirst} for zero in (b), and regularization scheme
\eqref{eq:contSecond} in (c).} \label{fig:3:2}
\end{figure}
The controlled output behavior of \eqref{eq:CL_syst} is exemplary
visualized in Figure \ref{fig:3:2} (a). Here a periodic
disturbance $d(t)=50\sin(20\pi t)$ is used while the control gain
is assigned as $\gamma=100$. The most simple first-order forward
Euler method (solver) is used. For numerical solvability of the
closed-loop system dynamics \eqref{eq:CL_syst}, two implementation
schemes are used. The first one
\begin{equation}
u(t)=\begin{cases}
-|x_{1}(t)|^{-1}\bigl(\gamma x_{1}(t)+|x_{2}(t)|x_{2}(t)\bigr) & |x_{1}(t)| \geq \mu\\
-\gamma\,\text{sign}\bigl(x_{1}(t)\bigr) & |x_{1}(t)| < \mu
\end{cases},\label{eq:contFirst}
\end{equation}
incorporates a threshold value $\mu
> 0$ for realizing numerical zero. The corresponding control
signal is shown opposite the disturbance value in Figure
\ref{fig:3:2} (b). For the sake of clarity, the same control
signal but low-pass (LP) filtered with 3000 Hz cut-off frequency
is also depicted. The second implementation scheme uses a
regularization factor $\mu > 0$, cf. \cite{ruderman2022}, so that
\begin{equation}
u(t) = -\Bigl( \bigl|x_{1}(t)\bigr| + \mu \Bigr)^{-1}\Bigl(\gamma
x_{1}(t)+\bigl|x_{2}(t)\bigr|x_{2}(t)\Bigr).\label{eq:contSecond}
\end{equation}
The corresponding control signal is shown opposite the disturbance
value in Figure \ref{fig:3:2} (c). Note that for both
implementation schemes, a low numerical value $\mu = 10^{-9}$ is
assumed.

\section{Experimental control results}

\label{sec:4}

The proposed control \eqref{eq:CL_syst} is experimentally
evaluated on the electro-magneto-mechanical actuator system with
one translational degree of freedom, see Figure \ref{fig:4:1}. The
system dynamics can be reduced to the model \eqref{eq:system}.
\begin{figure}[!h]
\centering
\includegraphics[width=0.5\columnwidth]{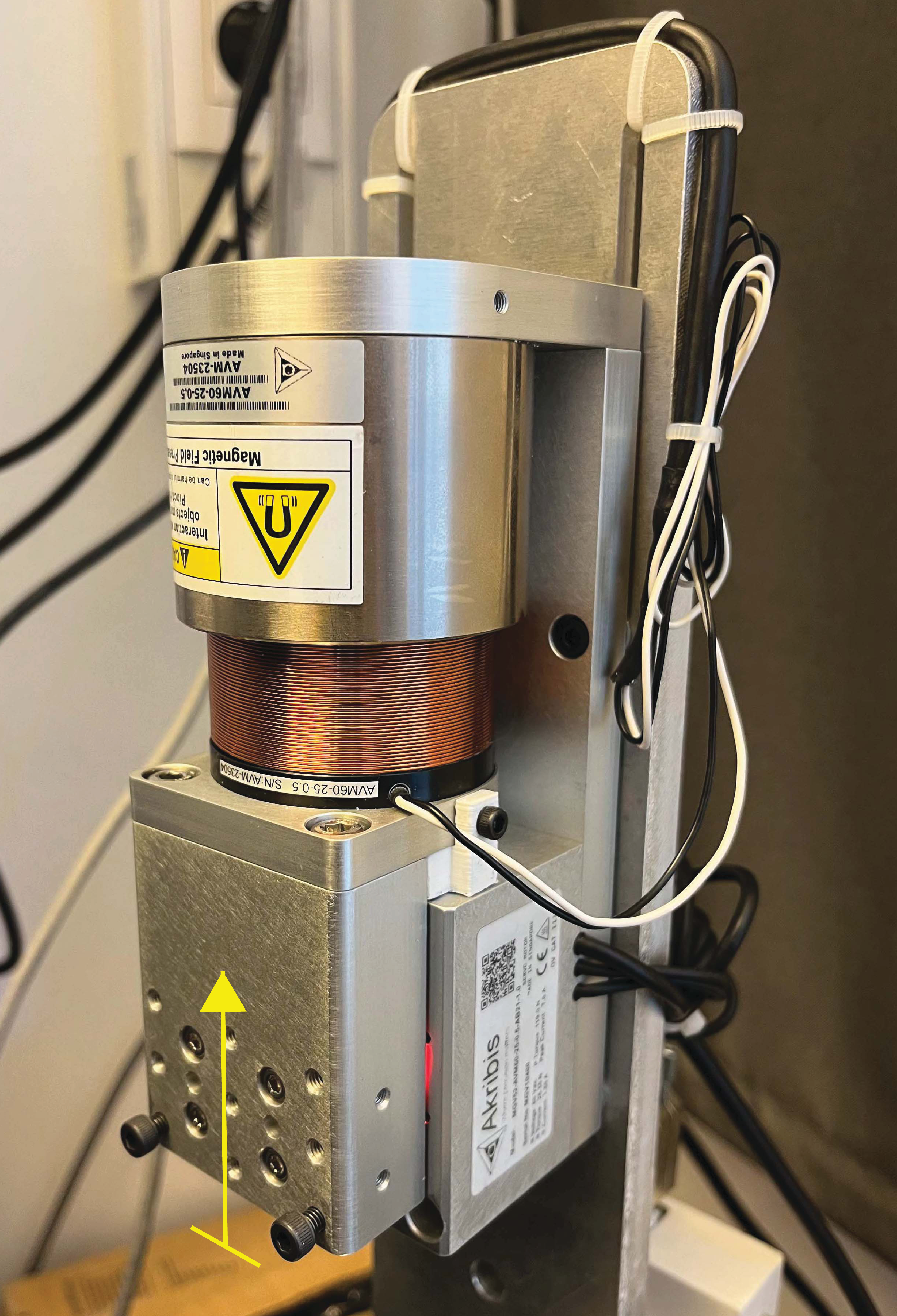}
\caption{\label{fig:4:1} Experimental setup of the actuator
system.}
\end{figure}
Technical details, including the identified system model and most
relevant hardware parameters of the setup, can be found in
\cite{ruderman2022}. The sampling rate is $10$ kHz. We note that
the overall upper-bounded perturbation $d$ is due to the constant
gravity term, but also due to the time-varying friction and
cogging force by-effects of the electro-magneto-mechanical
actuator.

The step response of the control \eqref{eq:CL_syst} (denoted
further as NLG owing to the nonlinear feedback gains) is compared
once with a standard proportional-derivative (PD) control and once
with the sub-optimal nonlinear damping control
\cite{ruderman2021a} (denoted further as SOND).

The benchmarking linear PD control
\[
u(t)_{\mathrm{PD}}=-\gamma x_{1}-\sigma x_{2},
\]
is designed first so that the unperturbed closed-loop control
system has two critically damped real poles, i.e., with an optimal
value of $\sigma$ that corresponds to the unity damping ratio.
This control is denoted by $\mathrm{PD}_{1}$. Then, its output
feedback gain $\gamma$ is enhanced while allowing also for the
conjugate complex pole pair with a corresponding transient
overshoot. That PD-control realization is denoted by
$\mathrm{PD}_{2}$.

The parameters of all four evaluated controllers are listed in
Table \ref{tab:1}.
\begin{table}[!h]
\global\long\def\arraystretch{1.5}%
\caption{\label{tab:1} Evaluated controllers}
\centering{}{ }%
\begin{tabular}{|p{2.5cm}||p{1cm}|p{1cm}|p{1cm}|p{1cm}|}
\hline { parameter \textbackslash\hspace{0.1mm} control }  & {
$\mathrm{PD}_{1}$ }  & { $\mathrm{PD}_{2}$ } & { $\mathrm{SOND}$ }
& { $\mathrm{NLG}$ }\tabularnewline \hline \hline { $x_{1}$-gain
$\gamma$ } & { $500$ }  & { $750$ }  & { $750$ } & { $25$
}\tabularnewline \hline { $x_{2}$-gain $\sigma$ }  & { $2$ }  & {
$4$ }  & { $-$ }  & { $-$ }\tabularnewline \hline
\end{tabular}
\end{table}

Note that for a fair comparison, the SOND control is assigned with
the same gaining factor $\gamma=750$ as the faster PD-control
$\mathrm{PD}_{2}$. The feedback gain of the NLG-control
$\gamma=25$ is intentionally set an order of magnitude lower than
the feedback gains of the PD and SOND controllers. This is for
demonstrating that the NLG-control can behave less aggressive
during the transient response and, at the same time, achieve a
much higher accuracy at steady-state. All experimentally measured
control responses, to the same reference step of
$y_{\mathrm{ref}}=8$ mm, are shown opposite each other in Figure
\ref{fig:4:2}.
\begin{figure}[!h]
\centering
\includegraphics[width=0.98\columnwidth]{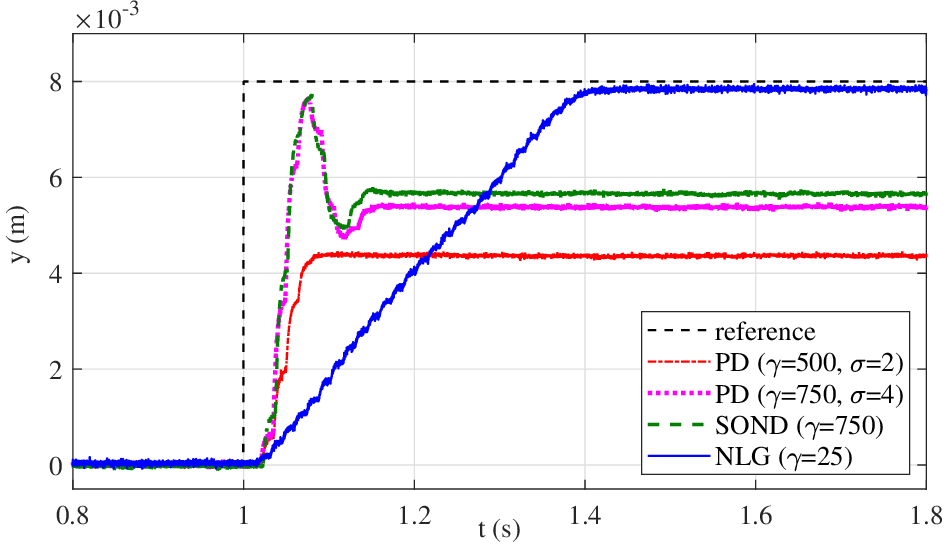}
\caption{\label{fig:4:2} Measured step response of PD, SOND, and
NLG controllers.}
\end{figure}

From the response of the NLG-control, one can recognize a nearly
uniform behavior of $x_{2}(t)$ during the whole transient phase,
cf. Figure \ref{fig:2:1} (a). Note that the corresponding
inclination of the $x_{1}(t)$ slope depends on the assigned
$\gamma$ parameter. It is also worth noting that a residual
non-zero, however minor, steady-state error of the NLG-control is
due to the sensor bias and noise, while the latter is additionally
impacting the used $x_{2}(t)$ signal, which is obtained via
differentiation and filtering of the measured $x_{1}(t)$ value.
The experimentally obtained signal $x_{2}(t)$ (when using the NLG
controller) is also shown in Figure \ref{fig:4:3} for the sake of
clarity.
\begin{figure}[!h]
\centering
\includegraphics[width=0.98\columnwidth]{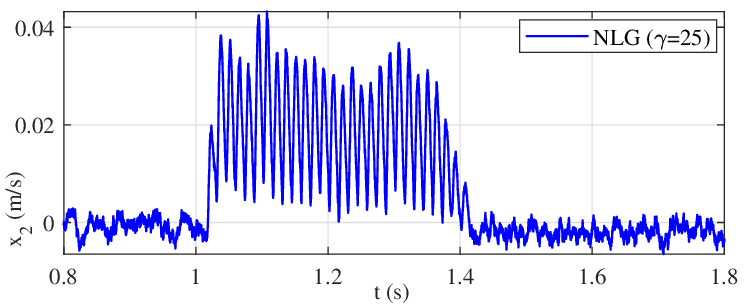}
\caption{\label{fig:4:3} The velocity state $x_{2}(t)$ of the NLG
control experiment.}
\end{figure}

\section{Summary}

\label{sec:5}

A novel robust nonlinear state feedback control for perturbed
second-order systems was introduced. The control law combines the
ideas of the state-dependent nonlinear damping introduced in
\cite{ruderman2021a} and the sliding mode control theory, which
equips the overall feedback law with an output-sign-dependent term
able to compensate for bounded perturbations. The resulted control
has only one tuning parameter, the overall output feedback gain,
whose selection should be based on the upper bound of the
perturbation magnitude, and this closed analytic dependence (even
though conservative) was shown in the paper. We provided a
detailed analysis of uniform boundedness and finite-time
convergence of the control system based on the proposed strict
Lyapunov function. The control signal appears to be continuous
during the transients, apart from the convergence instant, where
the state value reaches the origin (similarly to some high order
sliding mode control algorithms). Remarkable features of the
proposed control are that: (i) it is equivalent (from the
structural viewpoint) to a standard linear PD feedback control
and, at the same time, (ii) it can guarantee (without additional
integral action) for full compensation of the properly bounded
matched disturbances in the second-order systems without
overshoot, and arbitrary bounded in the steady state, while (iii)
providing a global finite-time convergence. The proposed control
was already evaluated experimentally on an
electro-magneto-mechanical actuator system with considerable
bounded disturbances. Analysis of its robustness abilities against
process and measurement noise and development of the methods for
its discretization can be considered as directions of future
research.

\section*{Acknowledgments}
\label{sec:6} This work was supported by AURORA mobility program (RCN grant number 340782).

\bibliographystyle{elsarticle-num}
\bibliography{references}

\vspace{10mm}

\appendix
\numberwithin{equation}{section}

Consider a second-order system
\begin{eqnarray}
\frac{\mathrm{d}x_{1}(\tau)}{\mathrm{d}\tau} & = & \bigl|x_{1}(\tau)\bigr|x_{2}(\tau),\label{eq:aux}\\
\frac{\mathrm{d}x_{2}(\tau)}{\mathrm{d}\tau} & = & -\gamma
x_{1}(\tau)-\bigl|x_{2}(\tau)\bigr|x_{2}(\tau)+\bigl|x_{1}(\tau)\bigr|d(\tau),\nonumber
\end{eqnarray}
in the time $\tau\geq0$, with the state
$x(\tau)=(x_{1}(\tau)\;x_{2}(\tau))^{\top}\in\mathbb{R}^{2}$ and
an essentially bounded input $d(\tau)\in\mathbb{R}$ with
$\|d\|_{\infty}\leq D$ for a given $D>0$, and with a gain
$\gamma>D$. The right hand-side of \eqref{eq:aux} is locally
Lipschitz continuous in $x(\tau)$ and $d(\tau)$, hence, the
solutions are uniquely defined for any initial condition
$x(0)\in\mathbb{R}^{2}$ and a measurable essentially bounded input
$d$ at least locally in time. Moreover, it is easy to check that
the origin is the equilibrium of the system and the submanifold
$\mathcal{X}=\{x\in\mathbb{R}^{2}:x_{1}=0\}$ is invariant
uniformly in $d$.
\begin{lem}
\label{lem:aux} The origin for \eqref{eq:aux} is uniformly
globally asymptotically stable if
\begin{equation}
\gamma>D^{1.5}+D+\frac{1}{2}.\label{eq:cond_gamma}
\end{equation}
\end{lem}
\begin{proof}
Let us consider the following Lyapunov function candidate
\begin{align*}
V(x) & =\frac{1}{2}\xi^{\top}(x)\left[\begin{array}{cc}
2\gamma & \varepsilon\\
\varepsilon & 1
\end{array}\right]\xi(x),\\
\xi(x) & =\left[\begin{array}{c}
\sqrt{|x_{1}|}\text{sign}(x_{1})\\
x_{2}
\end{array}\right],
\end{align*}
which is positive definite providing that the parameter
$\varepsilon\in(0,\sqrt{2\gamma})$. The function $V$ is
continuously differentiable for all
$x\in\mathbb{R}^{2}\setminus\mathcal{X}$. Note that
$V(x)=\frac{x_{2}^{2}}{2}$ for $x\in\mathcal{X}$ (it is an
invariant subspace for the system, and the trajectories of
\eqref{eq:aux} stay on it while being defined) and it is
straightforward to check that (we denote $V(\tau)=V(x(\tau))$)
\[
\frac{\mathrm{d}V(\tau)}{\mathrm{d}\tau}=-\bigl|x_{2}(\tau)\bigr|^{3}\leq0,
\]
implying convergence of $x_{2}(\tau)$ to zero. For
$x(0)\notin\mathcal{X}$ (note that
$\left(\mathbb{R}^{2}\setminus\mathcal{X}\right)\cup\{0\}$ is also
invariant for \eqref{eq:aux}), we obtain
\begin{gather*}
\frac{\mathrm{d}V(\tau)}{\mathrm{d}\tau}=-(\gamma-\text{sign}(x_{1}(\tau))d(\tau))\varepsilon\sqrt{|x_{1}(\tau)|}|x_{1}(\tau)|-\bigl|x_{2}(\tau)\bigr|^{3}\\
+\varepsilon\sqrt{|x_{1}(\tau)|}x_{2}^{2}(\tau)\left(\frac{1}{2}-\text{sign}(x_{1}(\tau))\text{sign}(x_{2}(\tau))\right)\\
+\bigl|x_{1}(\tau)\bigr|x_{2}(\tau)d(\tau)\\
\leq-(\gamma-D)\varepsilon\sqrt{|x_{1}(\tau)|}|x_{1}(\tau)|-\bigl|x_{2}(\tau)\bigr|^{3}\\
+\varepsilon\sqrt{|x_{1}(\tau)|}x_{2}^{2}(\tau)\left(\frac{1}{2}-\text{sign}(x_{1}(\tau))\text{sign}(x_{2}(\tau))\right)\\
+D|x_{1}(\tau)||x_{2}(\tau)|\\
\leq-[(\gamma-\frac{1}{2}-D)\varepsilon-\frac{2}{3}D^{1.5}]\sqrt{|x_{1}(\tau)|}|x_{1}(\tau)|\\
-(\frac{2}{3}-\varepsilon)\bigl|x_{2}(\tau)\bigr|^{3},
\end{gather*}
where on the last step the following inequalities have been
utilized (that are obtained from Young's inequality
\cite{young1912}):
\begin{gather*}
\sqrt{|x_{1}(\tau)|}x_{2}^{2}(\tau)\leq\frac{1}{3}\sqrt{|x_{1}(\tau)|}|x_{1}(\tau)|+\frac{2}{3}\bigl|x_{2}(\tau)\bigr|^{3},\\
D|x_{1}(\tau)||x_{2}(\tau)|\leq\frac{2}{3}D^{1.5}\sqrt{|x_{1}(\tau)|}|x_{1}(\tau)|+\frac{1}{3}|x_{2}(\tau)|^{3}.
\end{gather*}
Hence, for $\gamma>D+\frac{1}{2}$ and
\[
\frac{2}{3}\frac{D^{1.5}}{\gamma-\frac{1}{2}-D}<\varepsilon<\min\biggl\{\frac{2}{3},\sqrt{2\gamma}\biggr\}
\]
the desired property
\[
\frac{\mathrm{d}V(\tau)}{\mathrm{d}\tau}<0
\]
while $V(\tau)\ne0$ is obtained. The interval for the values of
$\varepsilon$ is not empty provided that the condition
\eqref{eq:cond_gamma} is verified. Indeed, this condition implies
that $\gamma>0.5$ for any $D>0$, hence,
$\min\{\frac{2}{3},\sqrt{2\gamma}\}=\frac{2}{3}$ and
$\frac{2}{3}\frac{D^{1.5}}{\gamma-\frac{1}{2}-D}<\frac{2}{3}$ is
equivalent to \eqref{eq:cond_gamma}. Combining the derived
estimates for $\frac{\mathrm{d}V(\tau)}{\mathrm{d}\tau}$ on the
sub-manifold of $\mathcal{X}$ and outside, the uniform global
asymptotic stability property of the origin for \eqref{eq:aux}
follows the conventional results \cite{Lin1996,khalil2002}.
\end{proof}

\end{document}